\documentclass{article}
\usepackage[utf8]{inputenc}

\usepackage{amssymb,amsmath, amsthm, amscd,ifthen}

\usepackage{tikz}
\usepackage{tikz-3dplot}

\newcommand{\RightAngle}[4][5pt]{%
	\draw ($#3!#1!#2$)
	--($ #3!2!($($#3!#1!#2$)!.5!($#3!#1!#4$)$) $)
	--($#3!#1!#4$) ;
}

\newcommand{\R}{{\mathbb R}}
\newtheorem{opr}{Definition}
\newtheorem{thm}{Theorem}
\newtheorem{gypo}[thm]{Conjecture}
\newtheorem{lem}[thm]{Lemma}

\newtheorem{prop}[thm]{Proposition}
\newtheorem{prb}[thm]{Problem}
\newtheorem{obs}[thm]{Observation}
\DeclareMathOperator{\inter}{int}

\date{}
\title{Proof of Schur's conjecture in $\R^d$ }

\begin{document}
\author{Andrey B. Kupavskii\footnote{Ecole Polytechnique F\'ed\'erale de Lausanne, Moscow Institute of Physics and Technology; Email: {\tt kupavskii@yandex.ru} \ \ Research supported in part by the Swiss National Science Foundation Grants 200021-137574 and 200020-14453 and by the grant N 15-01-03530 of the Russian Foundation for Basic Research.}, Alexandr Polyanskii\footnote{Moscow Institute of Physics and Technology, Mathematics Department, Technion — Israel Institute of Technology; Email: {\tt alexander.polyanskii@yandex.ru}\ \ Research suppoted in part by the Presedent Grant MK-3138.2014.1.}}

\maketitle
\begin{abstract} In this paper we prove Schur's conjecture in $\R^d$, which states that any diameter graph $G$ in the Euclidean space $\R^d$ on $n$ vertices may have at most $n$ cliques of size $d$. We obtain an analogous statement for diameter graphs with unit edge length on a sphere $S^d_r$ of radius $r>1/\sqrt 2$. The proof rests on the following statement, conjectured by F. Mori\'c and J. Pach: given two unit regular simplices $\Delta_1,\Delta_2$ on $d$ vertices in $\R^d$, either they share $d-2$ vertices, or there are vertices $v_1\in \Delta_1,v_2\in \Delta_2$ such that $\|v_1-v_2\|>1$. The same holds for unit simplices on a $d$-dimensional sphere of radius greater than $1/\sqrt 2$.
\end{abstract}

\vskip+0.2cm
\noindent\textbf{Keywords:} Schur's conjecture, diameter graph, unit simplices in $\R^d$.\\

\section{Introduction}
One of the classical problems in discrete geometry, raised by P. Erd\H os in 1946 \cite{Erd1}, is the following: given $n$ points in the plane, how many unit distances they may determine?
  The key definition related to the question of P. Erd\H os is that of a {\it unit distance graph}. A graph $G$ is a {\it unit distance graph} in $\R^d$ if its set of vertices is a finite subset of $\R^d$ and the edges are formed by the pairs of vertices which are at unit distance apart. In terms of distance graphs the question is to determine the maximal number of edges in a planar unit distance graph on $n$ vertices. In this paper we focus on the questions of this type for {\it diameter graphs}.
 A graph $G=(V,E)$ is a {\it diameter graph} in $\R^d$, if $V\subset \R^d$ is a finite set of diameter 1, and edges of $G$ are formed by the pairs of vertices that are at unit distance apart.

Diameter graphs arise naturally in the context of the finite version of the famous Borsuk's problem (see, e.g., \cite{BMP, Rai3} for the survey on Borsuk's problem), which is stated as follows: is it true that any (finite) set of unit diameter in $\R^d$ can be partitioned into $d+1$ subsets of strictly smaller diameter? The finite version is equivalent to the following question concerning diameter graphs: is it true that any diameter graph $G$ in $\R^d$ satisfies $\chi(G)\le d+1$?

A question about diameter graphs analogous to the question from the first paragraph has a simple answer: any set of $n$ points in the plane generates at most $n$ diameters, or any diameter graph on $n$ vertices in the plane has at most $n$ edges. This was proved by H. Hopf and E. Pannwitz in \cite{PH}. Interestingly, this result leads to a simple proof of the fact that Borsuk's question for finite sets in the plane have a positive answer. Indeed, it is easy to derive combinatorially that any graph $G$ on $n$ vertices with at most $n$ edges and such that any of its subgraphs has at least as many vertices as edges satisfies $\chi(G)\le 3$.
A. V\'azsonyi conjectured that any diameter graph in $\R^3$ on $n$ vertices has at most $2n-2$ edges. Again, it is easy to see that Borsuk's conjecture for finite sets in $\R^3$ follows from this statement. V\'azsonyi's conjecture was proved independently by B. Gr\"unbaum \cite{GR}, A. Heppes \cite{Hep2} and S. Straszewicz \cite{St}. An interesting generalization of this result to the case of $k$-th diameters was obtained by F. Mori\'c and J. Pach \cite{PP}.

While the maximum number of edges in a diameter graph in $\R^2,\R^3$ is linear in the number of vertices,  it becomes quadratic already in $\R^4$. To put the discussion in a more general context, we introduce the following notations.
Denote by $D_d(l,n)$ ($U_d(l,n)$) the maximum number of cliques of size $l$ in a diameter (unit distance) graph on $n$ vertices in $\R^d$. P. Erd\H os \cite{Erd1,Erd3} studied $U_d(2,n)$ and $D_d(2,n)$ for different $d$. He showed that for $d\ge 4$ we have $U_d(2,n), D_d(2,n) = \frac{\lfloor d/2\rfloor-1}{2\lfloor d/2\rfloor}n^2+\bar o(n^2).$ K. Swanepoel \cite{Swan} determined $U_d(2,n)$ for fixed even $d\ge 6$ and sufficiently large $n$ depending on $d$ and determined
 $D_d(2,n)$ for $d\ge 4$ and sufficiently large $n$.

 Functions $D_d(l,n),$ $U_d(l,n)$ for $l>2$ and similar functions were studied in several papers. In particular, the following conjecture was raised in \cite{Sch}:

 \begin{gypo}[Schur et. al., \cite{Sch}]\label{gypsp}
 We have $D_d(d,n) = n$ for $n\ge d+1$.
 \end{gypo}
This was proved by H. Hopf and E. Pannwitz for $d=2$ in \cite{PH} and for $d=3$ by Z. Schur et. al. in \cite{Sch}. In the latter paper the authors also proved that $D_d(d+1,n)=1$. In \cite{Philip2} P. Mori\'c and J. Pach progressed towards resolving this conjecture. Namely, they showed that Schur's conjecture holds in the following special case:

\begin{thm}[Theorem 1 from \cite{Philip2}]\label{thPh}
Given a diameter graph $G$  on $n$ vertices in $\R^d$, the number of $d$-cliques in $G$ does not exceed $n$, provided that any two $d$-cliques share at least $d-2$ vertices.
\end{thm}

As it turns out, Schur's conjecture and related questions are tightly connected with analogous questions for spherical sets. The spherical analogues were studied in a few papers. In particular, in the paper \cite{KP} V. Bulankina et al. noted that the statement of Theorem \ref{thPh} holds for spheres of large radii: given a diameter graph $G$  on $n$ vertices in a $d$-dimensional sphere $S^d_r$ with radius $r>1/\sqrt 2$, the number of $d$-cliques in $G$ does not exceed $n$, provided that any two $d$-cliques share at least $d-2$ vertices (Theorem 4 from \cite{KP}). Moreover, they showed that Schur's conjecture holds for $S_r^3$ for $r>1/\sqrt 2$. To be precise, we formulate Schur's conjecture for spheres separately:

 \begin{gypo}[Schur's conjecture for spheres]\label{gypsph}
Any diameter graph $G$ on $n$ vertices (and with edges of unit Euclidean length) on a sphere $S_r^d$ with $r>1/\sqrt 2$ has at most $n$ $d$-cliques.
 \end{gypo}

In the paper \cite{Kup10} A. Kupavskii studied properties of  diameter graphs in $\R^4$, in particular proving the four-dimensional Schur's conjecture. The following theorem completes the description of the quantity $D_4(l,n)$ for different $l$:

\begin{thm}[Theorem 5 from \cite{Kup10}]\label{th10}
\begin{enumerate}~\\
\item For $n\ge 52$ we have $$D_4(2,n)=\begin{cases}\lceil n/2\rceil \lfloor n/2 \rfloor + \lceil n/2\rceil +1, \ \text{ if } n\not\equiv 3\ \mathrm{mod}\ 4, \\
                    \lceil n/2\rceil \lfloor n/2 \rfloor + \lceil n/2\rceil, \ \ \ \ \  \ \text{ if } n\equiv 3\ \mathrm{mod}\ 4. \\
                     \end{cases}$$ (In Corollary 3 from \cite{Swan} the same was proved for sufficiently large $n$.)
\item  For all sufficiently large $n$ we have
 $$D_4(3,n)= \begin{cases}(n-1)^2/4+ n, \ \ \ \ \ \text{ if } n\equiv 1\ \mathrm{mod}\ 4, \\
                    (n-1)^2/4+n-1, \text{ if } n\equiv 3\ \mathrm{mod}\ 4, \\
                    n(n-2)/4+ n, \ \ \ \ \ \! \text{ if } n\equiv 0\ \mathrm{mod}\ 2. \end{cases}$$
\item(Schur's conjecture in $\R^4$)  For all $n\ge 5$ we have $D_4(4,n) = n$.
\end{enumerate}
\end{thm}

In \cite{Kup10} the first author also studied diameter graphs on $S^3_r$ with $r>1/\sqrt 2$. In particular, he showed that an analogue of V\'azsonyi's conjecture holds for diameter graphs on spheres.

In the next section we present our main results and discuss related questions. In Section \ref{sec3} we introduce some basic objects that are used in the proof. In Section \ref{sec4} we present the proofs of the results.

\section{New results and discussion}\label{sec2}
The main result of this paper is the proof of Schur's conjecture both in the Euclidean space and on the sphere in any dimension:

\begin{thm}\label{thmain} Schur's conjecture holds  \\
1. In the space $\R^d$,\\
2. On the sphere $S^d_r$ of radius $r>1/\sqrt 2.$
\end{thm}

The proof of the first part actually relies heavily on the second part, so the questions for the Euclidean space and for the sphere are indeed interconnected.\\

\textbf{Remark.} Note that throughout the article by a $k$-simplex in $\R^d$ we mean a set of $k+1$ vertices in $\R^d$ in general position.\\

Next we discuss several questions mentioned in the paper \cite{Philip2}. Since the authors of \cite{Philip2} proved Theorem \ref{thPh}, they naturally raised the following problem:

\begin{gypo}[F. Mori\'c and J. Pach, Problem 1 from \cite{Philip2}]\label{ph1} Any two unit regular simplices on $d$ vertices in $\R^d$ must share at least $d-2$ vertices, provided the diameter of their union is 1.
\end{gypo}

We confirm this conjecture (and its spherical version) in our paper, which together with Theorem \ref{thPh} and its spherical analogue from \cite{KP}, mentioned in the previous section, gives us the proof of Schur's conjecture both in the space and on the sphere. Another problem the authors of \cite{Philip2} raised deals with irregular simplices.

\begin{gypo}[Conjecture 3 from \cite{Philip2}]\label{ph2} Let $a_1\ldots,a_d$ and $b_1\ldots,b_d$ be two simplices on $d$ vertices in $\R^d$ with $d\ge 3$, such that all
their edges have length at least 1. Then there exist $i, j \in \{1,\ldots, d\}$ such that $\|a_i-b_j\|\ge 1$.
\end{gypo}

By slightly modifying the proof of Theorem \ref{thmain} it is not difficult to obtain the following theorem:
\begin{thm}\label{thirr} Consider a regular unit simplex $\{a_1,\ldots, a_d\}$ and a simplex  $\{b_1,\ldots,\\ b_d\}$ in $\R^d$ (or on $S^d_r$ with $r>1/\sqrt 2$), where the second simplex satisfies the property $\|b_i-b_j\|\ge 1$ for $i\neq j$. Then either these two simplices share $d-2$ vertices, or $\|a_i-b_j\|>1$ for some $i,j\in\{1,\ldots, d\}$.
\end{thm}

This theorem solves Conjecture \ref{ph2} in a stronger form in the case where one of the two simplices is regular.  We omit the proof, but the main additional ingredient needed is that the radius of the smallest ball that contains the simplex $b_1,\ldots, b_d$ is at least as big as for a regular unit $(d-1)$-simplex, provided that $\|b_i-b_j\|\ge 1$ for $i\neq j$. This, in turn, is an easy application of Kirszbraun's theorem (see \cite{Ak} for a short and nice proof):

\begin{thm}[Kirszbraun's theorem] Let $U$ be a subset of $X$, where $X$ is $\R^d$, $S^d$ or $H^d$ (a $d$-dimensional hyperbolic space). Then any nonexpansive map $f : U \to X$ can be extended to a nonexpansive map $f': X\to X$. A nonexpansive map $f:Y\to X$ is a map which satisfies $\|f(a)-f(b)\|\le \|a-b\|$ for any $a,b\in Y$.
\end{thm}

After having prepared the first version of this paper, we came across a paper by H. Maehara \cite{MaeS}, in which the author studies a seemingly unrelated concept of \textit{sphericity} of a graph: given a graph $G$, the \textit{sphericity} of $G$ is the minimum dimension in which the vertices of the graph can be represented as unit spheres in such a way that two spheres intersect (or touch) iff the corresponding vertices are connected by an edge. In  \cite{MaeS} the author discusses the sphericity of complete bipartite graphs. And, as it turned out, the main result of the paper is, in fact, the proof of  Conjecture \ref{ph2}, which was given 20 years before the conjecture was formulated! For a bit more on Maehara's result in the context of Schur's conjecture see Section \ref{secp1}.

Finally, in the paper \cite{Philip2} the authors raised the following general problem:

\begin{prb}[Problem 6 from \cite{Philip2}]\label{ph3} For a given $d$, characterize all pairs $k,l$ of integers such that for any set of $k$ red and $l$ blue points in $\R^d$ we can choose a red point $r$ and a blue point $b$ such that $\|r-b\|$ is at least as large as the smallest distance between two points of the same color.
\end{prb}

 For $k=d+1$ and $l=\lfloor\frac{d+1}2\rfloor$ it is not difficult to construct an example of two regular unit simplices in $\R^d$ on $k$ and $l$ vertices respectively, such that the distance between any two vertices from different simplices is smaller than 1, which we describe at the end of the next section. (An analogous, but somewhat different, example appeared in the latter version of the paper \cite{Philip2}.) We think that this is the extremal example, thus, we conjecture the following.

\begin{gypo}\label{gypo11} Given two unit simplices in $\R^d$, one on $d+1$ vertices, the other on $\lfloor\frac{d+1}2\rfloor +1$ vertices, either they share a vertex, or the diameter of their union is strictly larger than 1.
\end{gypo}
In \cite{KP2} the authors proved Conjecture \ref{gypo11} for $d=4$.

\section{Preliminaries}\label{sec3}

Given a hyperplane $\pi$, we  denote by $\pi^+$ and $\pi^-$ two closed half-spaces (half-spheres in the spherical case) that are determined by $\pi$.

The following object is very important for understanding diameter graphs:

\begin{opr}\label{opr1}
{\emph A} Reuleaux simplex $\Delta$ in $\R^d$  \emph{is a set formed by the intersection of the balls $B_i = B^{d}_1(v_{i})$ of unit radius with centers in $v_{i},\ i=1,\ldots, d+1$, where $v_{i}$'s are the vertices of a unit simplex in $\R^d$. In the case $d=3$ we call this object a \textit{Reuleaux tetrahedron}, and in the case $d=2$ we call it a} Reuleaux triangle.
\end{opr}

 We denote the $(d-1)$-dimensional spheres of unit radii with centers in $v_1,\ldots, v_{d+1}$ (the boundary spheres of $B_1,\ldots, B_{d+1}$) by $S_1,\ldots S_{d+1}$. A Reuleaux simplex is a spherical polytope, so one can naturally partition the boundary of a Reuleaux simplex into spherical faces of different dimensions: the vertices of the underlying simplex are the zero-dimensional faces, the arcs that connect the vertices are the one-dimensional faces and so on. We discuss it in more details a bit later in this section.
 The analogous definition could be given in the case of $S^d_r$, $r>1/\sqrt 2$. In this case we call the body a \textit{spherical Reuleaux simplex}. The only thing one has to keep in mind is that on a $d$-dimensional sphere we still consider spherical Reuleaux simplices on $d+1$ vertices. Note that, by Jung's theorem, on a $d$-dimensional sphere of radius $r = \sqrt{(d+1)/(2d+4)}$ one can have a regular unit $(d+2)$-simplex, which is, however, impossible for other radii (and, in particular, impossible for $r > 1/\sqrt 2$).

 For a given set $W$ we denote its interior by $\inter W$. In the paper we use several times the following simple observation.

 \begin{obs}\label{obs0} Consider two $d$-balls $B,B'$ in $\R^d$ of radii $r,r'$, correspondingly. Denote their boundary spheres by $S, S'$. Assume that $r>r'$ and that $S$ and $S'$ intersect in a $(d-2)$-dimensional sphere. Denote the hyperplane that contains $S\cap S'$ by $\sigma$ and assume that the centers of $S, S'$ are in the same closed halfspace $\sigma^+$ with respect to $\sigma$. The other halfspace we denote by $\sigma^-$. Then $B\cap\sigma^+\supset B'\cap\sigma^+$ and $B\cap\sigma^-\subset B'\cap\sigma^-$. Moreover, $\inter(B\cap\sigma^+)\supset B'\cap\inter(\sigma^+)$ and $B\cap\inter(\sigma^-)\subset \inter(B'\cap\sigma^-).$
 \end{obs}

 \begin{center}  \includegraphics[width=60mm]{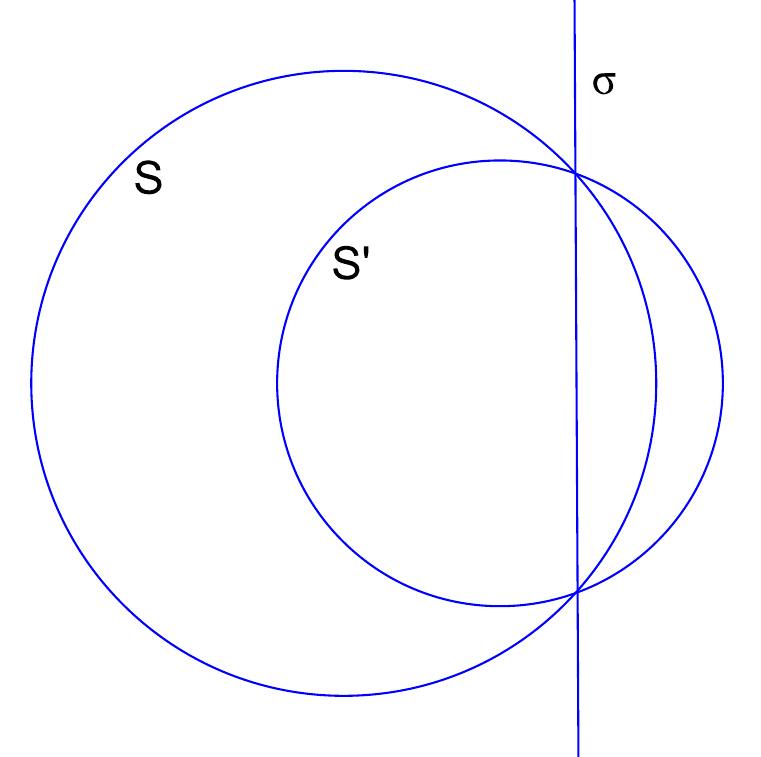}  \end{center}\begin{center}  Figure 1. Ball $B$ contains the part of $B'$ to the left from  $\sigma$. \end{center}

See Fig. 1, illustrating the observation. We apply the observation above to deduce the following two lemmas.
 \begin{lem}\label{lem0} Consider a Reuleaux simplex $\Delta\subset \R^d$ with the set of vertices $v_{i},\ i=1,\ldots, d+1$ and a ball $B$ with a boundary sphere $S$, circumscribed around the $d$-simplex $\{v_1,\ldots, v_{d+1}\}$. Then $\Delta$ lies inside $B$, moreover, $\Delta\cap S = \{v_1,\ldots, v_{d+1}\}.$
 \end{lem}

 \begin{proof} Denote by $\pi_i$ the hyperplane that passes through the vertices $v_1,\ldots,v_{i-1},\\ v_{i+1},\ldots, v_{d+1}$. We denote the closed halfspace defined by $\pi_i$ and that contains $v_i$ by $\pi_i^+$. This halfspace contains the convex hull of the vertices  $\{v_1,\ldots, v_{d+1}\}$. By $\pi_i^-$ we denote the other halfspace. Applying  Observation \ref{obs0} to the balls $B_i$, $B$, we get that $\inter B\supset \inter(B\cap \pi_i^-)\supset B_i\cap \inter(\pi_i^-)$. After going through all possible values of $i$, we get that $$\inter B\supset \cup_{i=1}^{d+1}(B_i\setminus\pi_i^+)\supset (\cap_{i=1}^{d+1}B_i)\setminus (\cap_{i=1}^{d+1}\pi_i^+) = \Delta\setminus (\cap_{i=1}^{d+1}\pi_i^+).$$
On the other hand, $\cap_{i=1}^{d+1}\pi_i^+$ is just a convex hull of the points $\{v_1,\ldots, v_{d+1}\}$, and it is for sure contained in $B$, moreover, it intersects $S$ only in its vertices $v_1,\ldots, v_{d+1}$.
 \end{proof}

\begin{lem}\label{lem01} Consider a Reuleaux simplex $\Delta\subset \R^d$ with the set of vertices $v_{i},\ i=1,\ldots, d+1$. Then the intersection of $\Delta$ with the hyperplane $\pi$ that passes through $v_1,\ldots,v_d$ is a Reuleaux simplex $\Delta'$ with vertices $v_i, i =1,\ldots, d$.
\end{lem}
\begin{proof}
We have $\Delta' = \pi \cap \bigcap_{i=1}^{d+1} B_i $ and we have to prove that $\Delta' = \pi \cap \bigcap_{i=1}^{d} B_i$. Thus, it is enough to show that $\pi \cap \bigcap_{i=1}^{d} B_i\subset B_{d+1}\cap \pi$. Denote the circumscribed ball of the Reuleaux simplex $\pi \cap \bigcap_{i=1}^{d} B_i$ by $B'$ and its boundary sphere by $S'$. Then $B_{d+1}\cap \pi = B'$, and the statement of the lemma follows from the first claim of Lemma \ref{lem0}.
\end{proof}

 \begin{center}  \includegraphics[width=60mm]{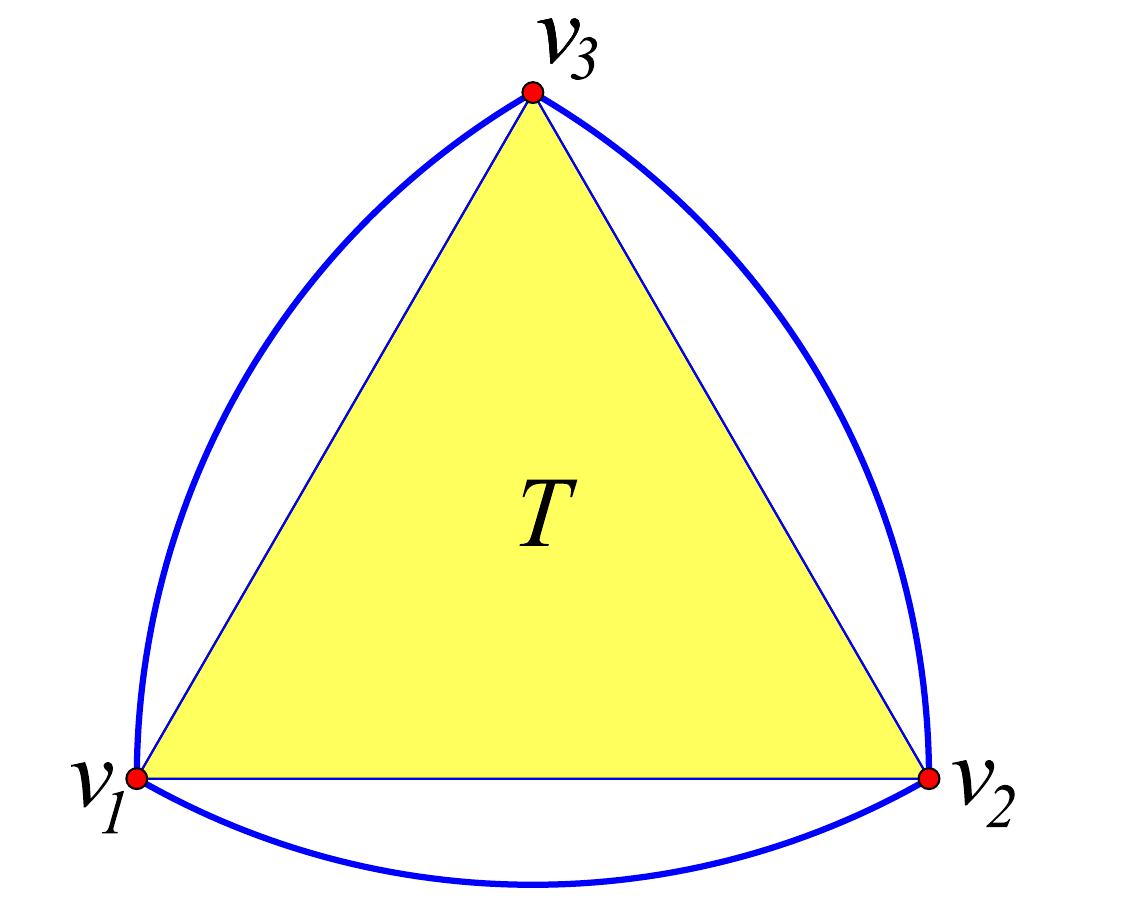}  \end{center}\begin{center}  Figure 2. A simplex, its convex hull $T$ and its Reuleaux simplex $\Delta$. \end{center}

We need some knowledge about the structure of a Reuleaux simplex $\Delta$ as a spherical polytope. Let the vertices of $\Delta$ be $v_1,\ldots, v_{d+1}$ and their convex hull be denoted by $T$ (see Fig. 2, where $T$ is the shaded triangle and $\Delta$ is the set bounded by the arcs). The boundary of the Reuleaux simplex $\Delta$ can be partitioned into relatively open spherical regions of different dimensions in the following way. Consider all spheres $U$ that are formed as intersections of several spheres out of $S_1,\ldots, S_{d+1}$. For example, the intersection of the first $d$ spheres is a two-point set, with one of its points being $v_{d+1}$. We only exclude the intersection of all the spheres, which is empty. We denote the set of all such spheres $U$ by $\mathcal S$. For each point on the boundary of $\Delta$ we may find the sphere from $\mathcal S$ of minimal dimension that contains it. The set of all points from the boundary of $\Delta$ that correspond to a given $U\in \mathcal S$ we call a \textit{face}. It is defined by the set of strict quadratic inequalities and, thus, is naturally a relatively open set.

The center of each sphere $U\in \mathcal S$ coincides with the center of some of the faces of $T$. Namely, if $U = \cap_{j=1}^l S_{i_j}$, where $1\le i_1<i_2<\ldots<i_l\le d+1$, then the center of $U$ is the centerpoint of the $(l-1)$-dimensional face with the vertices $v_{i_1},\ldots, v_{i_l}$. On the other hand, $U$ contains all the vertices from $\{v_1,\ldots,v_{d+1}\}\setminus\{v_{i_1},\ldots, v_{i_l}\}$ and, since $U$ is a $(d-l)$-dimensional sphere and points $\{v_1,\ldots,v_{d+1}\}$ are in general position, these points determine $U$. We call these points the \textit{vertices of the face}.

Each face of $\Delta$ is a connected set, moreover, it is obtained from the face of $T$ with the same set of vertices via projection from the center $O$ of $T$ on the boundary of $\Delta$. We verify this property in what follows. Fix a vertex set of the face $F$ of $\Delta$. W.l.o.g., it is $\{v_1,\ldots, v_i\}$. The points of $F$ on the boundary of $\Delta$ are the ones that, first, are at distance 1 from $v_{i+1},\ldots, v_{d+1}$ and, second, are at distance strictly less than 1 from $v_1,\ldots, v_i$. From the first condition we have $F\subset S_{i+1}\cap\ldots\cap S_{d+1}\subset \alpha$,  where the flat $\alpha$ is an affine hull of the points $O, v_1,\ldots, v_i$ (we may replace $O$ by the centerpoint of the face $\{v_{i+1},\ldots, v_{d+1}\}$ of $T$). Since $Ov_j$ has the same length for all $j$, the second condition is equivalent to the fact that each point $w$ from $F$ satisfies the angular inequality $\angle wOv_j\le \angle wOv_{i+1}$ for all $j=1,\ldots, i$. Each of these inequalities is satisfied in a halfspace, defined by the hyperplane that passes through $O$ and vertices $\{v_1,\ldots, v_{d+1}\}\setminus \{v_j,v_{i+1}\}$. These halfspaces  bound the face $\{v_1,\ldots, v_i\}$ in $T\cap \alpha$ and, as we have showed above, they also bound the face with the same set of vertices in $\Delta\cap\alpha$. Moreover, they pass through $O$, which concludes the proof of the statement about the central projection. We formulate the findings of the last three paragraphs in a lemma:

\begin{lem}\label{lem02} In the notations introduced above, consider a Reuleaux simplex $\Delta\subset \R^d$. Then its boundary may be split into relatively open connected spherical regions (faces), each of which corresponds to an intersection of several spheres out of $S_1,\ldots, S_{d+1}$. The face $F$ of $\Delta$ that corresponds to the intersection of the spheres $S_{i+1},\ldots,S_{d+1}$ lies on the sphere with the center in the centerpoint $C$ of the face $F'$ of $T$ with vertices $v_{i+1},\ldots, v_{d+1}$: $C = (v_{i+1}+\ldots, +v_{d+1})/(d-i+1)$. The flat of minimal dimension that contains $F$ is an affine hull of points $O,v_1,\ldots, v_i$.  Moreover, the face $F$ is equal to the central projection of the convex hull of $\{v_1,\ldots,v_i\}$ from the center of $\Delta$ to the boundary of $\Delta$.
\end{lem}

Next we define the object which is of a particular importance for the paper:

\begin{opr}\label{opr1}
A rugby ball $\Theta$ in $\R^d$  \emph{is a set formed by the intersection of the balls $B_i = B^{d}_1(v_{i})$ of unit radius with centers in $v_{i},\ i=1,\ldots, d$, where $v_{i}$'s are the vertices of a unit $(d-1)$-simplex in $\R^d$.}
\end{opr}
We omit the analogous definition of a \textit{spherical rugby ball}.
Note the difference between a Reuleaux simplex and a rugby ball. The latter is an intersection of $d$ balls instead of $d+1$ for the former. The intersection of the hyperplane $\pi$ that passes through $v_1,\ldots,v_d$ and the corresponding rugby ball is a Reuleaux simplex of codimension 1. The rugby ball is symmetric with respect to $\pi$.

Consider a Reuleaux simplex $\Delta$ on the vertices $v_1,\ldots, v_{d+1}$, the rugby ball $\Theta$ on the vertices $v_1,\ldots, v_d$, and the hyperplane $\pi$ containing vertices $v_1,\ldots,v_d$. Suppose that $v_{d+1}\in \pi^+.$ 

\begin{lem}\label{lem03}
In the notations introduced above, we have $\Delta^+:= \Delta\cap \pi^+ = \Theta\cap\pi^+$.
\end{lem}
\begin{proof} Since $\Delta\cap \pi^+ = \Theta\cap\pi^+\cap B_{d+1},$ it is obviously sufficient to show that $B_{d+1}\cap\pi^+\supset \Theta\cap\pi^+.$ Consider a ball $B$ circumscribed around $\Delta$. By using the same argument as in Lemma \ref{lem0}, we get that $B\cap\pi^+\supset \Theta\cap\pi^+$. On the other hand, applying Observation \ref{obs0} to $B, B_{d+1}$, we get that $B_{d+1}\cap\pi^+\supset B\cap\pi^+$.
\end{proof}

Now we describe the construction mentioned in the end of the previous section. Take a regular simplex on $d+1$ vertices in $\R^d$ as the set of red points. Next, construct the Reuleaux simplex on the the red points and choose $l = \lfloor\frac{d+1}2\rfloor$ midpoints $y_1,\ldots, y_l$ of some $l$ pairwise disjoint arcs (1-dimensional faces) that connect the vertices of the Reuleaux simplex. It could be checked that the distance between the midpoints of two such arcs is strictly bigger than 1. To see this, one have to consider a coordinate representation of the simplex $\{v_1,\ldots, v_{d+1}\}$ in the hyperplane $x_1+\ldots,+x_{d+1} = 1$ in $\R^{d+1}$ and calculate the coordinates of such a midpoint.  Thus, if we consider the simplex on $y_1,\ldots, y_l$ and contract it a little, we will get a simplex on vertices $x_1,\ldots, x_l$ with all vertices inside the Reuleaux simplex and with all sides greater than 1. We take $\{x_i\}$ as the set of blue points, which together with the red points gives us the desired example.

\section{Proof}\label{sec4}
\subsection{Reduction to an auxiliary theorem}\label{secp1}

The proof of Theorem \ref{thmain} involves an inductive argument based on the following auxiliary theorem, which is of interest by itself:

\begin{thm}\label{thcom} Given a diameter graph $G$\\
1. In the space $\R^d$, $d\ge 3$;\\
2. On the sphere $S^d_r$ of radius $r>1/\sqrt 2,$ $d\ge 3$,\\
any two $d$-cliques in $G$ must share a vertex.
\end{thm}

As we came across the paper \cite{MaeS} we were thinking whether or not to try to give a proof of Theorem \ref{thcom} using Maehara's result. The Euclidean case of Theorem \ref{thcom} is almost equivalent to  \cite[Theorem 2]{MaeS} (which is equivalent to Conjecture \ref{ph2}). However, to get part 1 of Theorem \ref{thcom} from Conjecture \ref{ph2}, one has to replace the non-strict inequality on distances by a strict one. It turns out that this seemingly technical detail is not easy to overcome. If applied directly, \cite[Theorem 2]{MaeS} gives only that, if there are two regular unit $(d-1)$-simplices in a diameter graph in $\R^d$ that do not share a vertex, then there must be at least one edge between them. This is clearly not sufficient (and it is fairly easy to reduce Theorem \ref{thcom} to this case). The proof of Maehara, however, may be modified to give a proof of Theorem \ref{thcom} in the Euclidean case, but becomes significantly more complicated. Maehara's technique relies heavily on linear algebra, thus we do not even know if it is possible to apply his techniques in the spherical case.  We believe that in sum it would not simplify the proof of Theorem \ref{thcom}. Therefore, we decided to leave the proof as it is and not to utilize Maehara's ideas.\\

In this subsection we describe how to derive Theorem \ref{thmain} from Theorem \ref{thcom}.
Consider two $d$-cliques in a diameter graph $G$ in $\R^d$ (or on $S_r^d$ with $r>1/\sqrt 2$). Then, by Theorem \ref{thcom}, these two cliques must share a vertex. All the remaining vertices of the two simplices must lie on the $(d-1)$-dimensional unit sphere $S$ with the center in the common vertex of the two simplices. The vertices on $S$ form two $(d-2)$-dimensional unit simplices, and, since the subgraph that lies on $S$ is a diameter graph, we can again apply Theorem \ref{thcom} and obtain that the two $(d-2)$-dimensional simplices on $S$ must share a vertex, which gives the second common vertex for the $d$-cliques.

Finally, we obtain that any two $d$-cliques must share $d-2$ common vertices and apply a spherical analogue of Theorem \ref{thPh}, which was proved in \cite{KP}. This completes the proof of Schur's conjecture. We only have to verify the following: the spheres that we work with during this process always have radius greater than $1/\sqrt 2$. This was shown to be true in \cite[Lemma 4]{KP}. We state this fairly easy lemma and present its proof for completeness.

\begin{lem}\label{lemrad} Consider a  $d$-dimensional sphere  $S = S^d_r$ of radius $r> 1/\sqrt 2$ and a unit simplex  $\Delta$ on $k$ vertices $v_1,\ldots, v_k$ with all its vertices on $S$. Then the intersection  $\Omega$ of the sphere $S$ and the $k$ unit spheres with centers in $v_1,\ldots, v_k$ is a sphere of radius  $r_{\Omega}> 1/\sqrt 2$.
\end{lem}
\begin{proof} We assume that the sphere is embedded into a Euclidean space, and we work in that space. Denote by $v= \frac 1k\sum_{i=}^k v_i$ the center of the sphere $S'$, circumscribed around $\Delta$. By Jung's theorem,  the radius $r'$ of $S'$ is equal to $\sqrt{\frac{k-1}{2k}}$. So, the radius $r''$ of the sphere $S''$, which is the intersection of  $k$ unit spheres with centers in $v_1,\ldots, v_k$ is $\sqrt{1 - \frac{k-1}{2k}} = \sqrt{\frac{k+1}{2k}}.$ Note that the center of $S''$ is also $v$. Denote the center of $S$    by $O$. Then the center $w$ of $\Omega$ lies on the segment $Ov$ of length $b$. Since $v_1,\ldots, v_k$ lie on $S$, we have
$b^2 = r^2-(r')^2 = r^2-\frac{k-1}{2k}$. Suppose $w$ splits the segment $Ov$ into the parts of length  $b-a,a$ respectively. Then, since $\Omega\subset S$, we get $r_{\Omega}^2 = r^2-(b-a)^2$. We also have $\Omega\subset S''$ so we get $r_{\Omega}^2 = \frac{k+1}{2k}-a^2.$ Therefore,

$$2r_{\Omega}^2  = r^2-b^2+\frac{k+1}{2k}+2ab-2a^2 =\frac{k-1}{2k}+\frac{k+1}{2k}+2a(b-a)> 1,$$
because it is easy to see that $a, b-a>0$.
\end{proof}

\subsection{Sketch of the proof of Theorem \ref{thcom}}
Our main goal is to prove Theorem \ref{thcom}. The proof of this theorem also goes by induction. The base case $d = 3$ is known to be true (it was verified for $\R^3$ in \cite{Dol} and for $S^3_r$ with $r>1/\sqrt 2$ in \cite{Kup10}). We reduce the problem for the $d$-dimensional space or for the $d$-dimensional sphere to the analogous problem for the $(d-1)$-dimensional sphere. Since the base case is already verified, this concludes the proof of the theorem.

To justify the induction step, we proceed as follows. We consider two unit simplices on $d$ vertices $K_1,K_2$ and build a rugby ball $\Theta$ around $K_1$. We analyze the possible positions of the vertices of $K_2$ with respect to $\Theta$, in particular, with respect to the plane $\pi$ that passes through $K_1$.

The first step is to prove Lemma \ref{lemimp}, which, roughly speaking, tells us that if we have two vertices in one of the halfspaces $\pi^+,\pi^-$, one of which projects inside the convex hull $T\subset \pi$ of $K_1$, then $K_1$ and $K_2$ have common vertices. This already reduces a lot the work to be done.

The second step, which is considered in (i) of the next subsection, is to consider the case when none of the vertices of $K_2$ are projected strictly inside $T$. In that case we use the fact that then all these vertices lie inside a full-dimensional ball $B$, circumscribed around $K_1$ and with the center in the center of $K_1$. This is due to the fact, stated in Observation \ref{obs01}, that $\Theta$ is contained in the union of $B$ and the set of points of $\Theta$ that project inside $T$. Once we know that $K_2$ is contained in $B$, we conclude that it is contained in the boundary $S$ of $B$, which leads to a conclusion that all the vertices are actually projected inside $T$. This leads to the same conclusion as in the previous step.

Finally, if none of the two are true, we have just one vertex $w_1$ in $\pi^+$, which projects inside $T$, and the rest in $\pi^-$. This correspond to the case (ii) of the next subsection. We describe a procedure that rotates $K_1$ and creates more and more unit distances between $w_1$ and the vertices of $K_1$. At the end this results in having a full-dimensional regular unit simplex $K_1\cup \{w_1\}$. Then all the vertices of $K_1\cap K_2$ apart from $w_1$ lie on a unit sphere centered at $w_1$ and we can apply induction.

It is tempting to give a unified proof of Theorem \ref{thcom}, in which the Euclidean and the spherical cases are both treated at the same time. But, on the other hand, it makes the proof more difficult to understand. So we chose a intermediate option. We give two separate proofs, first for the Euclidean case and then we describe the differences and peculiarities of the spherical case in a separate subsection. The key lemma (Lemma \ref{lemimp}), however, has a unified proof.

\subsection{Proof of Theorem \ref{thcom}. Euclidean case}\label{secp2}

We begin with the following important lemma:

\begin{lem}\label{lemimp} Take a Reuleaux simplex $\Delta$ in $\R^d$ and the hyperplane $\pi$ containing the vertices $v_1,\ldots,v_d$ of $\Delta$. Consider the body $\Delta^+$. Suppose $v,w\in \Delta^+$, and suppose that the projection $v'$ of $v$ on the hyperplane $\pi$ lies inside the convex hull $T$ of $v_1,\ldots, v_d$. Then $\|v-w\|\le 1,$ with the equality possible in the following two cases: 1. One of the vertices $v, w$  coincides with one of $v_1,\ldots, v_d$. 2. The vertex $w$ lies in the hyperplane $\pi$ on the border of a Reuleaux simplex $\Delta_{\pi}$, constructed on the vertices $v_1,\ldots, v_d$. At the same time the projection $v'$ of the vertex $v$ on the hyperplane $\pi$ must lie on $\partial T$.
\end{lem}

We defer the proof of this lemma until Section \ref{sec45}, where we give a unified proof in both Euclidean and spherical cases.\\

Consider a diameter graph $G$ and two $d$-cliques $K_1$, $K_2$ in $G$. Denote by $v_1,\ldots, v_d$ the vertices of $K_1$. Form a rugby ball $\Theta$ on $K_1$ and denote  the hyperplane containing $K_1$ by $\pi$. The following step is essential for the proof. Consider a $d$-dimensional ball $B$, circumscribed around the clique $K_1$. It has a center in the center $O$ of the clique $K_1$ and radius $\sqrt{(d-1)/2d}$, but has one dimension more than a normal circumscribed ball in the hyperplane $\pi$. Denote the boundary sphere of $B$ by $S$. As usually, denote the $(d-1)$-dimensional spheres of unit radii with centers in $v_1,\ldots, v_d$ (the boundary spheres of $B_1,\ldots, B_d$) by $S_1,\ldots S_d$.

The set $S\cap S_i$ for any $i=1,\ldots,d$ is a sphere that lies in the hyperplane $\pi_i$ orthogonal to $\pi$. Indeed, it is true due to the fact that both centers ($O$ and $v_i$) lie in $\pi$. This together with Observation \ref{obs0} gives us the following crucial observation.

 \begin{obs}\label{obs01} In the notations introduced above, whenever a point lies in $\Theta\backslash \inter B$, its projection on the hyperplane $\pi$ falls inside the convex hull $T$ of $v_1,\ldots,v_d$. If a point lies in $\Theta\backslash B$, then its projection falls in $\inter T$.
\end{obs}

Suppose that there are at least two vertices $w_1,w_2$ of $K_2$ in $\pi^+\cap \Theta$. If one of them, say $w_1$, does not lie in $B$, then, by Observation \ref{obs01}, its projection on $\pi$ falls strictly inside $T$, and we are done by Lemma \ref{lemimp}. Indeed, checking the conditions in Lemma \ref{lemimp}  that allow $\|w_1-w_2\|=1$ to hold, one sees that condition 2 does not take place since the projection of $w_1$ falls strictly inside $T$, so the first condition must hold and, consequently, one of the vertices $w_1, w_2$ must coincide with one of the vertices of $K_1$. The same reasoning apply for any two points $w_1,w_2$ of $K_2$ lying in $\pi^-\cap \Theta$.

Now we are left with two possibilities. \\

\textbf{(i)} On both sides of the hyperplane $\pi$ we have at least two points of $K_2$, or all vertices of $K_2$ lie on one side. This case, which seems to be essential, actually has a short resolution. In this case all points from $K_2$ lie inside the ball $B$, and we are able to use some of its properties. Namely, we know that, since $K_2$ is a clique of size $d$, then the radius of the minimal ball that contains the clique equals the radius of $B$ (even though it has a smaller dimension). This means that the center of that minimal ball must coincide with $O$, and all the points of $K_2$ must in fact lie on $S$ (otherwise the minimal ball will have a smaller radius). This, by Observation \ref{obs01}, gives us that all the points of $K_2$ are projected inside $T$. Since there are at least $3$ vertices in $K_2$, then in one of the halfspaces, say, $\pi^+$, there are at least two vertices $w_1,w_2$ of $K_2$. We may then apply Lemma \ref{lemimp}. The next step is to check the two conditions from the lemma that allow the equality $\|w_1-w_2\|=1$ to hold.

Condition 1 gives that one of $w_1,w_2$ coincides with one of the vertices of $K_1$. Condition 2 gives that $w_1$ lies on $\pi$, therefore, $w_1\in\Theta\cap \pi\cap S$. We note that that $\Theta\cap \pi$ is a Reuleaux simplex in the hyperplane $\pi$ and $S\cap \pi$ is its circumscribed sphere. Thus, by Lemma \ref{lem0}, we conclude that  $\Theta\cap \pi\cap S= \{v_1,\ldots, v_d\}.$ Therefore, in any case one of $w_1,w_2$ coincide with some of $v_1,\ldots, v_d$ and we are done in Case (i). \\

\textbf{(ii)} The other possibility is that exactly one vertex, say, $w_1$, lies in $\pi^+$, while the others lie in $\pi^-$. Moreover, we may assume that $w_1\notin B$, otherwise, all the vertices of $K_2$ lie inside $B$ and we argue as in the previous case.

We treat this case as follows. We prove that the two cliques have a common vertex by contradiction. Assuming the contrary, we start with a configuration of two $d$-cliques $K_1,K_2$ without common vertices of the type described above. We perturb the first clique, obtaining a valid configuration of two simplices $K'_1, K_2$ without common vertices, provided that the initial configuration was valid. Thus, if we obtain a contradiction at some point, it means that the initial configuration was as well impossible.

We  try to perturb the simplex $K_1$ so that  $w_1$ will get to the top of the rugby ball $\Theta$, constructed on the perturbed $K_1$, or, in other words, that $w_1$ will form a regular unit simplex with the (possibly perturbed) vertices $v_1,\ldots,v_d.$ Note that we do not modify $K_2$.

Here is the procedure. Suppose the distance between $w_1$ and $v_1$ is strictly less than 1. We rotate $v_1$ around the vertices $v_2,\ldots,v_d$, which are fixed. The possible trajectory of $v_1$ is a circle, and we push $v_1$ towards $\pi^-$. Denote the image of $v_1$ by $v'$. We stop the rotation if one of the following two events happen:\\

\noindent \textbf{Event 1.} The distance between $v'$ and $w_1$ is equal to 1.\\
\textbf{Event 2.} Some of $w_2,\ldots, w_d$ fall on the hyperplane $\pi'$ that passes through $v', v_2,\ldots,v_d$. \\

Before analyzing these two possibilities, we have to make some preparations. We start with the following simple observation.

\begin{obs}\label{obs1} Consider two points $x,y$ in $\R^d$ and a hyperplane $\tau$ that passes through the middle of the segment $xy$ and is orthogonal to it. Denote by $\tau^+$ the closed halfspace bounded by $\tau$ and that contains $x$. Then for any point $z\in \tau^+$ we have $\|z-x\|\le \|z-y\|$.
\end{obs}

Next we formulate in a lemma that, while having valid configurations all the time, we eventually arrive at one of the two events above. Define $B'$ as the image of $B$ under the rotation, and similarly define $\Theta'$, $S'$, $S_1'$, and $K_1'$.

\begin{lem}\label{lem2facts} 1. The point $w_1$ stays inside $\Theta'$, moreover, $w_1\notin B'$.\\
2. No vertex among $w_2,\ldots, w_d$ can escape from $\Theta'\cap (\pi')^-$ without falling onto $\pi'$ first.\\
3. No vertex among $w_2,\ldots, w_d$ can coincide with some of the vertices of $K_1'$.
\end{lem}

\begin{proof}[ \textbf{Part 1. }]  The point $w_1$ stays inside $\Theta'$ since we do not move $v_2,\ldots, v_d$ and because of Event 1. In what follows we prove that $w_1\notin B'$.

Consider the hyperplane $\gamma$, which contains the intersection of $S$ and $S'$. It passes through the vertices $v_2,\ldots, v_d$ and through the middle of the segment $vv'$. Denote by $\gamma^+$ the halfspace that contains $v$. We apply Observation \ref{obs1} for the hyperplane $\gamma$ and the centers of balls $B,B'$, which are obviously symmetric with respect to $\gamma$. We get that  $B\cap \gamma^+\supset B'\cap\gamma^+$ and  $B\cap \gamma^-\subset B'\cap\gamma^-$.

Next, we show that $w_1\in \gamma^+$. This is due to the fact that any point from $\pi^+\cap \Theta\backslash B$ projects from above inside the convex hull $T$ of the vertices of the Reuleaux simplex $\Theta\cap\pi$ (see Observation \ref{obs01}). We have $T\subset \gamma^+$. Since  the rotation is made continuously, we may for sure assume that the angular distance between $v$ and $v'$ is less than $90^{\circ}$. Therefore, any point that is projected on $\pi$ inside $T$ from above, lies in $\gamma^+$ (see Fig. 3, where points that are projected inside $T$ from above, lie in the shaded rectangle).
From $\pi^+\cap \Theta\backslash B\subset \gamma^+$ it follows that $ \pi^+\cap \Theta\backslash B\subset \pi^+\cap \Theta\backslash B'$, and, consequently, $w_1\notin B'$.

\begin{center}  \includegraphics[width=60mm]{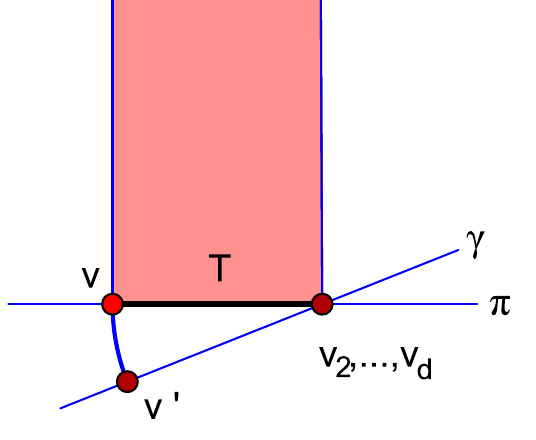}  \end{center}\begin{center}  Figure 3. $\pi^+\cap \Theta\backslash B$ projects inside $T$ and thus is inside $\gamma^+$.\end{center}

\noindent\textit{\textbf{ Part 2. }} This fact is proved in a similar fashion. Consider the hyperplane
that contains the intersection of $S_1$ and $S_1'.$ It is again $\gamma$, moreover, due to Observation \ref{obs1}, we have $\Theta'\cap\gamma^-\supset\Theta\cap\gamma^-$. Indeed, since the spheres $S_2,\ldots,S_d$ do not change, we  only have to look at the intersection of $B_1$ and $B_1'$, and we fall into a situation which is similar to the one considered in the previous part. Therefore, the only way for a point $w_i$ to escape $\Theta'$ is to fall onto $\gamma$ first. But this is not possible, because any position of $\gamma$ was a position of $\pi'$ at the earlier stage of rotation, so the point $w_i$ has to fall onto $\pi'$ first.

\noindent\textit{\textbf{ Part 3. }} Clearly, $w_i$ cannot coincide with $v_2,\ldots, v_d$ since none of them is moved during the rotation. Assume $w_i = v'$. But it means that before the rotation $w_i$ lay on the same arc as $v'$. However,  the projection of this arc on the hyperplane $\pi$ is a straight segment connecting the vertex $v_1$ and the center $O$. Therefore, the projection of $w_i$ onto $\pi$ falls inside $T$, and, since we have more than 1 vertex in $\pi^-$, we have a common vertex between $K_1$ and $K_2$ by Lemma \ref{lemimp}. This contradicts our assumptions.
\end{proof}

The only thing that is left to do is to analyze the two events given by the procedure.  Suppose that Event 2 happens, and the point $w_2$ from $K_2$ falls onto $\pi'$. Then we have two vertices of $K_2$ in $(\pi')^+$, and, since $w_2\ne \{v',v_2,\ldots, v_d\}$, we a common vertex of $K_1'$ and $K_2$ using Lemma \ref{lemimp}, a contradiction.

If Event 1 happens, then we take another vertex of $K_1$ and proceed in the same way. Finally, assume that we cannot continue the procedure and there is no common vertex of $K_1'$ and $K_2$. Denote by $v'_1,\ldots, v'_d$ the images of the vertices of $K_1$ after all the rotations. Then $w_1$ forms a unit $d$-simplex with  $v'_1,\ldots, v'_d$. In this case all the vertices $v'_1,\ldots,v'_d, w_2,\ldots,w_d$ lie on the unit sphere with center in $w_1$.

By induction, a unit $(d-1)$-simplex and a unit $(d-2)$-simplex on a $(d-1)$-dimensional sphere of radius greater than $1/\sqrt 2$ must share a common vertex. Therefore, we get a contradiction in any case and $K_1$, $K_2$ must have a common vertex in the first place.

\subsection{Proof of Theorem \ref{thcom}. Spherical case}\label{secp2}
\subsubsection{Preliminaries on spherical geometry}\label{sec441}
In Section \ref{secp2} we work on a $d$-dimensional sphere $\Gamma$ of radius greater than $1/\sqrt 2$.

Spherical geometry is very similar to Euclidean. To make the proof work in this case, one should, more or less, only change the notation: planes are changed to diametral (or great) spheres, halfspaces to hemispheres, balls to spherical caps. We will often use the Euclidean names for the spherical objects, e.g., say ``a plane'' instead of ``a diametral sphere''. This should not cause confusion, since we will mostly work in terms of internal spherical geometry. However, when it is convenient, we will think of the sphere as a subset of a Euclidean space, and interpret points of $\Gamma$ as vectors. In the next several paragraphs we will list the facts from spherical geometry that we will use in the proof. For an introduction to elementary spherical geometry we refer the reader to Chapter 1 of the book due to L. Fejes T\'oth \cite{LFT}. For a systematic treatment of spherical geometry, that by far covers all the material used in this paper we refer to \cite{Vin}.

\textbf{1.} There is a natural way to assign dimensions to spherical planes, such that the definition will work the same way as in the Euclidean case. Namely, the dimension of a diametral sphere is equal to the dimension of the minimal Euclidean plane that contains it. Note that a spherical line consists of two points.

\textbf{2.} For a flat (diametral sphere) $\gamma$ denote by $\gamma^*$ the maximal flat such that any vector from $\gamma^*$ is orthogonal to any vector in $\gamma$. If $\gamma$ is a hyperplane (diametral sphere of codimension 1), by $\gamma^+,\gamma^-$ we denote the closed half-spaces bounded by $\gamma$.

\textbf{3.} For a given hyperplane $\gamma$ and an arbitrary point $\Gamma\backslash \gamma^*$ we can define the projection of a point $v$ to $\gamma$. Consider the two-point set $\gamma^*$.  Then the projection $v'$ of $v$ on $\gamma$ is the closest intersection point to $v$ of the great circle that goes through $\gamma^*$and $v$ with the plane $\gamma$.

\textbf{4.} We define the reflection $R_{\gamma}$ with respect to a given  hyperplane $\gamma$. For any given point $v$ in $\Gamma\backslash \gamma^*$ we consider the great circle that contains $\gamma^*$ and $v$, and find a point $R_{\gamma}(v)$ on that circle, which is symmetric to $v$ with respect to the projection of $v$ on $\gamma$. As for the $\gamma^*$, the reflection interchanges the two points in $\gamma^*$.

\textbf{5.} Using reflections, it is easy to introduce the notion of orthogonality to a hyperplane, which would be convenient for us. Namely, a plane $\sigma$ is orthogonal to a hyperplane $\gamma$, if $R_{\gamma}(\sigma) = \sigma$.

\textbf{6.} Suppose that we have a $k$-sphere $\Omega$ on $\Gamma$, which is not diametral. It is easy to show that any such sphere is contained in a (spherical) plane $\gamma$ of dimension $k+2$. Indeed, taking a Euclidean point of view, for any $k$-sphere there is a $(k+2)$-dimensional plane that passes through the center of $\Gamma$ and contains $\Omega$. Its intersection with $\Gamma$ is the desired (spherical) plane. Note that this is the minimal plane that contains $\Omega$.

\textbf{7.} For points in an open hemisphere $\Gamma^+$ of $\Gamma$ one can easily define the distance between two points as the shorter angle between the corresponding vectors. In particular, the distance between any point in $\gamma$ and any point in $\gamma^*$ is $\pi/2$. We denote the spherical distance between $u_1,u_2\in \Gamma^+$ by $\rho(u_1,u_2)$.

\textbf{8.} We define an angle between the two intersecting arcs as the dihedral angle between the corresponding vector planes. For three distinct points $u_1,u_2,u_3 \in \Gamma^+$ we denote by $A(u_1,u_2,u_3)$ the angle between the arcs $u_1u_2$ and $u_2u_3$.

\textbf{9.} There is a version of Pythagoras' theorem for spherical triangles. Namely, given a right spherical triangle $u_1,u_2,u_3$ in $\Gamma^+$ with $A(u_1,u_2,u_3)=\pi/2,$ one have $\cos(\rho(u_1,u_3))=\cos(\rho(u_1,u_2))\cos(\rho(u_2,u_3)).$ Moreover, Pythagoras' theorem is a corollary of the spherical cosine law:
\begin{multline}\label{cos} \cos(\rho(u_1,u_3))=\cos(\rho(u_1,u_2))\cos(\rho(u_2,u_3))+\\ +\sin(\rho(u_1,u_2))\sin(\rho(u_2,u_3))\cos(A(u_1,u_2,u_3)).\end{multline} One can deduce the following statement out (\ref{cos}): suppose we are given three distinct points $u_1,u_2,u_3 \in \Gamma^+$ and the angle $A(u_1,u_2,u_3)$ between the arcs $u_1u_2$ and $u_2u_3$ is at least $\pi/2$. Assume moreover, that $\rho(u_1,u_2),\rho(u_3,u_2)$ are less than $\pi/2$. Then $\rho(u_1,u_3)>\max\{\rho(u_1,u_2),\rho(u_2,u_3)\}$.

\textbf{10.} We need the notion of a convex hull of a set of points $\{u_1,\ldots,u_k\}$. The straightforward way to define it is by using the Euclidean interpretation. It is simply the intersection of the cone formed by vectors corresponding to $u_1,\ldots,u_k$ and $\Gamma$. Note that the boundary of such a convex hull is formed by planes (diametral spheres).

\subsubsection{The proof}The distance in $\Gamma$, that corresponds to Euclidean distance 1, we denote by $\phi$. Note that, since in our case $\Gamma$ has radius greater than $1/\sqrt 2$, we have $\phi<\pi/2$.
Suppose we are given a diameter graph on $\Gamma$, which contains two simplices $K_1$ and $K_2$ on $d$ vertices. We consider the spherical rugby ball $\Theta$, formed by vertices of $K_1$, and the diametral sphere $\pi$ that contains $K_1$.

The proof stays almost the same as in the Euclidean case. We describe all the differences in what follows. All the notations are translated to this case from the Euclidean case.

First, we show that the spherical rugby ball $\Theta$ is contained in one of the open hemispheres of $\Gamma$. We denote such a hemisphere by $\Gamma^+$. Consider the unit ball $B_1$ with the center in $v_1$ (one of the vertices of $K_1$) . On one hand, since the radius of $\Gamma$ is bigger than $1/\sqrt2$, $B_1$ is contained in the open hemisphere $\Gamma^+$ with the center in $v_1$. On the other hand, surely, $B_1\supset \Theta$.

Consider a spherical Reuleaux simplex $\Delta$ and a point $x$ inside the (spherical) convex hull of its vertices. Then an open halfsphere with the center in $x$ contains $\Delta$. This is due to the fact that $\Delta$ is contained in the intersection of open halfspheres with centers in the vertices of $\Delta$, therefore, if we think about the vector interpretation of the situation, any vector $y$ from $\Delta$ has positive scalar products with any of the vectors representing the vertices. Since $x$ is a convex combination of the vectors representing the vertices, $x$ and $y$ have positive scalar product. We formulate it as the first part of the following observation:

\begin{obs}\label{obs2} 1. Consider a spherical Reuleaux simplex $\Delta$ on $\Gamma$ and a point $x$ that lies in the spherical convex hull of the vertices of $\Delta$. Then for any $y\in \Delta$ we have $\rho(x,y)<\pi/2$.\\
2. Consider any diametral hypersphere $\gamma$ on $\Gamma$ and any point $x$, $x\notin \gamma^*$. Let $x'$ be the projection of $x$ to $\gamma$. Then $\rho(x,x')<\pi/2$.
\end{obs}


The proof of Theorem \ref{thcom} starts with Lemma \ref{lemimp}. Its proof is deferred till the next subsection, and for now we assume that it holds in the spherical case as well. In what follows, we go through the changes in the remaining part of the proof.

\begin{obs}\label{obs3}
Spheres $S\cap S_i$ lie in the hyperplane $\pi_i$, which is orthogonal to $\pi$.
\end{obs}
\begin{proof} The first thing we note is that $S\cap S_i$ is a $(d-2)$-sphere that is not diametral. It would be diametral only if both $S$ and $S_i$ are diametral, which is not the case. Thus, by point 6 from the previous subsection, the minimal plane that contains $S\cap S_i$ is of codimension 1. Next, note that $R_{\pi}(S\cap S_i)=S\cap S_i.$ This is due to the fact that both centers of $S$ and $S_i$ lie in $\pi$. Hence, the same should hold for $\pi_i$, and by point 5 from the previous subsection we obtain the desired orthogonality.
\end{proof}

To conclude the proof in the case when there are either none or at least two vertices in each halfspace $\pi^+,\pi^-$ (corresponding to case (i) of the Euclidean proof), we use the same proof as in the Euclidean case. The ingredients that we add are the spherical version of Lemma \ref{lemimp}, Observation \ref{obs3} and the following observation, which shows that  the circumscribed ball considerations still work in the spherical case.

\begin{obs} Suppose we are given two unit simplices $K_1$, $K_2$ on $d$ vertices. Suppose $K_2$ lies inside the $d$-dimensional ball $B$ of diameter $f$, which is a ball of minimal diameter that contains $K_1$. Denote by $B_2$ the circumscribed ball for $K_2$. Then, if $B$ and $B_2$ do not coincide, the intersection $B\cap B_2$ is contained in a ball of diameter strictly smaller than $f$.
\end{obs}

\begin{proof} Choose an arbitrary point $u$ on the open segment connecting the centers $O, O_2$ of $B,B_2$, respectively. By the last property in point 9 from the previous subsection, combined with the fact that $\rho(x,u),\rho(u,O),\rho(u,O_2)<\pi/2$ for any $x \in B\cap B_2,$ we have $\rho(x,u)<\max\{\rho(x,O), \rho(x,O_2)\},$ since $A(O,u,x)$ or $A(O_2,u,x)$ is at least $\pi/2$.   We obtain that $K_2$ is contained in a ball of radius strictly smaller than $f$, which is impossible. Thus, the centers of $B$ and $B_2$ coincide.
\end{proof}
We are left to modify Case (ii) of the Euclidean proof, in which we have one vertex of $K_2$ in $\pi^+\cap\Theta,$ while the rest lie in $\pi^-\cap\Theta$.

Lemma \ref{lem2facts} works exactly as in the Euclidean case. To prove it, we note that the plane $\gamma$ satisfies the following equation: $R_{\gamma}(B') = B$, which is why $B\cap\gamma^+\supset B'\cap\gamma^+$. Indeed, $B'\cap\gamma^+ = R_{\gamma}(B\cap\gamma^-)$ and $B\cap\gamma^-$ is less than a halfball, while $B\cap\gamma^+$ is bigger than a halfball. Similar reasoning applies for the inclusion $\Theta'\cap\gamma^-\supset \Theta \cap\gamma^-$.

As for the third part of Lemma \ref{lem2facts}, consider the case when $w_2$ falls into $\pi'$ and, moreover, $w_2$ coincides with $v'$. We need to check that the arc on which $v', v_1$ lie projects inside the spherical convex hull $T$ of vertices $v_1,\ldots,v_d$. It is clear that the circle $S_2\cap\ldots\cap S_d$ and the sphere $S$ touch at $v_1$ (in the plane $\pi$). Thus  the circle lies in the exterior of the ball $B$ and the point $w_2$ lies in $\Theta\backslash B$. This by Observation \ref{obs01} gives that the arc projects inside $T$. Therefore, we can apply the spherical analogue of Lemma \ref{lemimp} to conclude the proof of Lemma \ref{lem2facts} in the spherical case.

The analysis of the possibilities work the same, and the proof of Theorem \ref{thcom} in the spherical case is complete.

\subsection{Proof of Lemma \ref{lemimp}}\label{sec45}
We give a unified proof of Lemma \ref{lemimp} in both Euclidean and spherical cases. We  use a unified terminology, in particular, $\rho(x,y)$ for both Euclidean and spherical distance and hyperplanes for both hyperplanes and hyperspheres. The distance between points forming edges in a diameter graph, as in the spherical case, we denote by $\phi$.\\

It is enough to consider the case when none of $v,w$  coincide with the vertices of $\Delta$.
Consider the projections $v', w'$ of $v,w$ on the hyperplane $\pi$ (see Fig. 4). We have two possibilities:\\

\begin{figure}
\tdplotsetmaincoords{70}{10}
\begin{tikzpicture}
[tdplot_main_coords, scale=4,
cube/.style={very thick,black},
grid/.style={very thin,gray},
axis/.style={->,blue,thick},
plane_pnt/.style={color=black!30!red,fill=black!30!red},
pnt/.style={color=black,fill=black}
]

\coordinate (A1) at (0,0,0);
\coordinate (A2) at (1,0,0);
\coordinate (A3) at (0.5,{sqrt(3)/2},0);
\coordinate (A4) at (0.5,{sqrt(1/12)},{sqrt(2/3)});

\coordinate (B12) at ($(A1)!0.5!(A2)$);
\coordinate (B23) at ($(A2)!0.5!(A3)$);
\coordinate (B13) at ($(A1)!0.5!(A3)$);
\coordinate (B14) at ($(A1)!0.5!(A4)$);
\coordinate (B24) at ($(A2)!0.5!(A4)$);
\coordinate (B34) at ($(A3)!0.5!(A4)$);

\coordinate (Pl)  at (1.3,0.8,0);

\coordinate (V)  at (0.2,0.2,0.3);
\coordinate (V') at (0.2,0.2,0);

\coordinate (W)  at (0.7,0.6,0.2);
\coordinate (W') at (0.7,0.6,0);

\def\margin{0.6}
\def\topx{1+\margin};
\def\botx{-\margin};
\def\topy{{sqrt(3)/2+\margin*3/4}};
\def\boty{-\margin};

\draw[thin,dashed,red] (\botx,\boty,0) -- (\topx,\boty,0);
\draw[thin,dashed,red] (\topx,\boty,0) -- (\topx,\topy,0);
\draw[thin,dashed,red] (\topx,\topy,0) -- (\botx,\topy,0);
\draw[thin,dashed,red] (\botx,\topy,0) -- (\botx,\boty,0);

\draw[red] (A1) -- (A2);
\draw[red] (A2) -- (A3);
\draw[red] (A3) -- (A1);

\draw[thin,dashed,red] (V') -- (W');
\draw[thin,dashed] (V') -- (V);
\draw[thin,dashed] (W') -- (W);

\draw[red,rotate=0] (A3) arc (180-60:180:1);
\draw[red,rotate=0] (A2) arc (0:60:1);
\draw[red,rotate=0] (A1) arc (-90-30:-90+30:1);

\pgfmathsetmacro\arcradius{sqrt(3)/2}
\pgfmathsetmacro\arcstart{asin(1/3)}

\tdplotsetthetaplanecoords{90}
\tdplotdrawarc[tdplot_rotated_coords]{(B12)}{\arcradius}{\arcstart}{90}{}{}
\tdplotsetthetaplanecoords{210}
\tdplotdrawarc[tdplot_rotated_coords]{(B23)}{\arcradius}{\arcstart}{90}{}{}
\tdplotsetthetaplanecoords{330}
\tdplotdrawarc[tdplot_rotated_coords]{(B13)}{\arcradius}{\arcstart}{90}{}{}

\RightAngle[2pt]{(V)}{(V')}{(W')};
\RightAngle[2pt]{(W)}{(W')}{(V')};

\def\pntsize{0.04};
\draw[plane_pnt] (A1) circle (\pntsize em);
\draw[plane_pnt] (A2) circle (\pntsize em);
\draw[plane_pnt] (A3) circle (\pntsize em);


\draw[plane_pnt] (V') circle (\pntsize em)	node[right] { $v'$ };
\draw[pnt] (V) circle (\pntsize em)	node[right] { $v$ };

\draw[plane_pnt] (W') circle (\pntsize em)	node[below left] { $w'$ };
\draw[pnt] (W) circle (\pntsize em)	node[left] { $w$ };
\draw[pnt] (A4) circle (\pntsize em);

\draw[plane_pnt] (Pl) 	node[below left] { $\pi$ };

\end{tikzpicture}
Figure 4
\end{figure}

1. $\rho(w',w)\ge \rho(v',v)$. Since $v'$ lies inside $T$ and the maximum of the distances from a fixed point to the points of a polytope is attained on the vertices of the polytope, there exists a vertex of $T$, say, $v_1$, such that $\rho(v_1,w')>\rho(v',w').$ Using Pythagoras theorem (in the spherical case, point 9 from Section \ref{sec441} and Observation \ref{obs2}), we get that $\phi\ge\rho(v_1, w)>\rho(v',w).$

 To complete the proof of Lemma \ref{lemimp} in this case we have to show that $\rho(v',w)\ge\rho(v,w)$. In the Euclidean case it follows right away from the fact that $v$ is closer to $\pi$ than $w$. In the spherical case we also have the same property, since $\pi/2>\rho(w',w)\ge \rho(v',v)$. It implies the inequality $A(w,v,v')\ge\pi/2$, and by point 9 from Section \ref{sec441} we get that $\rho(w,v')\ge\max\{\rho(w,v),\rho(v,v')\}$.\\

2. $\rho(w',w)< \rho(v',v)$. Assume for a moment that we obtained an inequality similar to the one in the previous case: there exists a vertex of $T$, say, $v_1$, such that $\rho(v_1,v')\ge \rho(v',w').$ Note that the inequality is not strict in this case. Then, arguing in a similar way, we first get that $$\phi\ge\rho(v_1, v)\overset{(1)}{\ge}\rho(v,w')\overset{(2)}{\ge}\rho(v,w).$$

Unfortunately, we do not have any strict inequality in this chain. However, it is clear that the inequality $(2)$ is strict unless $w = w'$. In the Euclidean case it is obvious, and in the spherical case it again follows from point 9 of Section \ref{sec441}. Moreover, it is easy to show that the inequality $(1)$ is strict unless $w$ is on the border of $\Delta_{\pi}$. Indeed, if $w$ is in $\Delta_{\pi}$, but not on the border, then change $w$ to the point of intersection of the ray $v'w$ with the border of $\Delta_{\pi}$. The distance between $v$ and $w$ will increase, which makes the inequality strict. This proves that $\rho(v,w)=\phi$ may only occur if the first part of condition 2 from the lemma is satisfied.

Therefore, to conclude the proof of Lemma \ref{lemimp} it is sufficient to show that such a vertex $v_1$ exists (and to verify that in the case $\rho(v,w)=\phi$ the other part of condition 2 from the lemma  holds as well). We formulate it as a separate lemma:

\begin{lem}\label{lemrelo} Consider a unit Reuleaux simplex $\Delta_{\pi}$ in $\R^{d-1}$ (or $S^{d-1}_r,$ $r>1/\sqrt 2$) with the set of vertices $v_1,\ldots, v_d$ and two points $v', w\in \Delta_{\pi}, $ different from the vertices of $\Delta_{\pi}$, so that $v'$ lies in the convex hull $T$ of $v_1,\ldots, v_d.$ Denote by $w'$ the projection of $w$ on $\pi$. Then there exists $i$ so that $\rho(v_i,v')\ge \rho(w',v')$, with the equality possible only in case if $v'$ lies on the boundary of $T$.
\end{lem}

Lemma \ref{lemrelo} is proved via a repetitive application of the proposition below.

\begin{prop}\label{lemred}
 Consider a closed half-space $\omega^+$ in $\R^d$ or $S^d$ bounded by a hyperplane $\omega$. Let  $\Upsilon$ be a sphere with center in $C$, where $C\in \omega$. Let  $\Omega$ be an open region on $\Upsilon$, $\Omega\subset \Upsilon \cap \omega^+.$ Consider two points
 $X\in \omega^+,$ $Y\in\Omega$. Then one can find a point $Y'\in \partial \Omega$ such that
\begin{itemize}
\item $\rho(X,Y)<\rho(X,Y')$, if $X\ne C$;
\item $\rho(X,Y)=\rho(X,Y')$, if $X=C$.
\end{itemize}
\end{prop}
\begin{center}  \includegraphics[width=80mm]{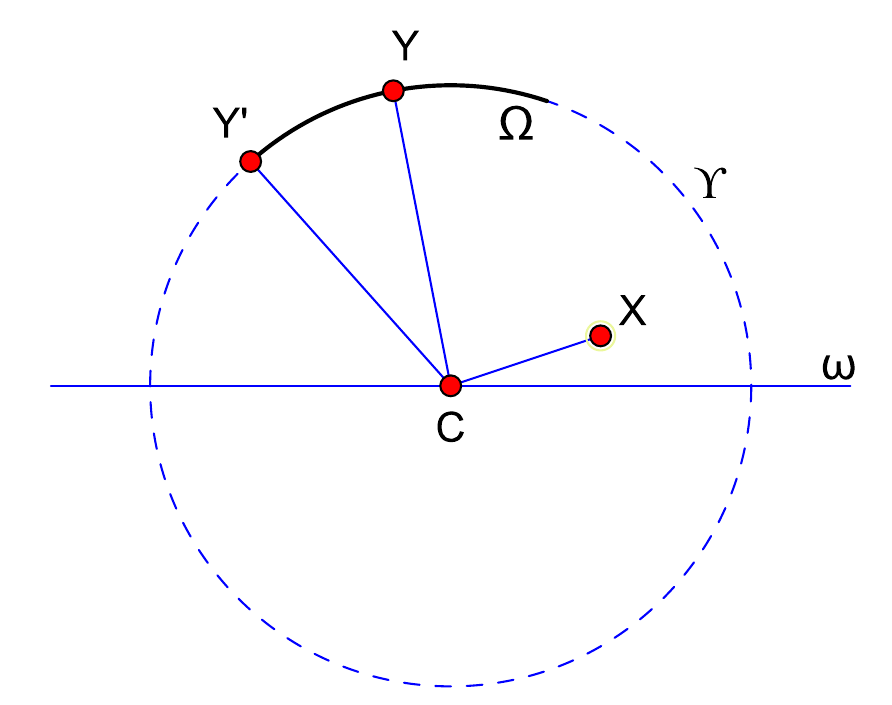}  \end{center}\begin{center}  Figure 5. $Y'$ is farther away from $X$ than $Y$. \end{center}

\begin{proof} The equality from the statement of the lemma is obvious since  $\Omega$ and $ \partial\Omega$ both lie on $\Upsilon$. As for the inequality, consider the two-dimensional plane $\gamma$ that contains the points $C,X,Y$ (See Fig. 5). The line $CY$ splits the plane into two closed halfplanes  $\gamma^+, \gamma^-$. Let $X\in \gamma^+$. There is at least one point $Y'\in \partial \Omega$ in $\gamma^-$, which is different from $Y$. Then we have the inequality for the angles $\angle XCY<\angle XCY'$ and, thus,  by the law of cosines, $\rho(X,Y)<\rho(X,Y').$ In the spherical case it holds because the first summand in (\ref{cos}) stays the same in both cases, while in the second one the only change is that the last multiple gets smaller. Note that we do not need to put any restrictions on $\rho(X,C),\rho(C,Y)$, as it is done in the end of point 9 of the previous subsection.
\end{proof}

\begin{proof}[Proof of Lemma \ref{lemrelo}] As we have already said, we apply Lemma \ref{lemred} repeatedly. We may assume that $w\in \partial\Delta_{\pi}.$ The boundary of a Reuleaux simplex can be partitioned into the open spherical regions of different dimensions, as it is stated in Lemma \ref{lem02}. We remark that considerations of Section \ref{sec3} transfer with some obvious changes to the spherical case. Thus, we have a spherical analogue of Lemma \ref{lem02}.

\begin{center}  \includegraphics[width=100mm]{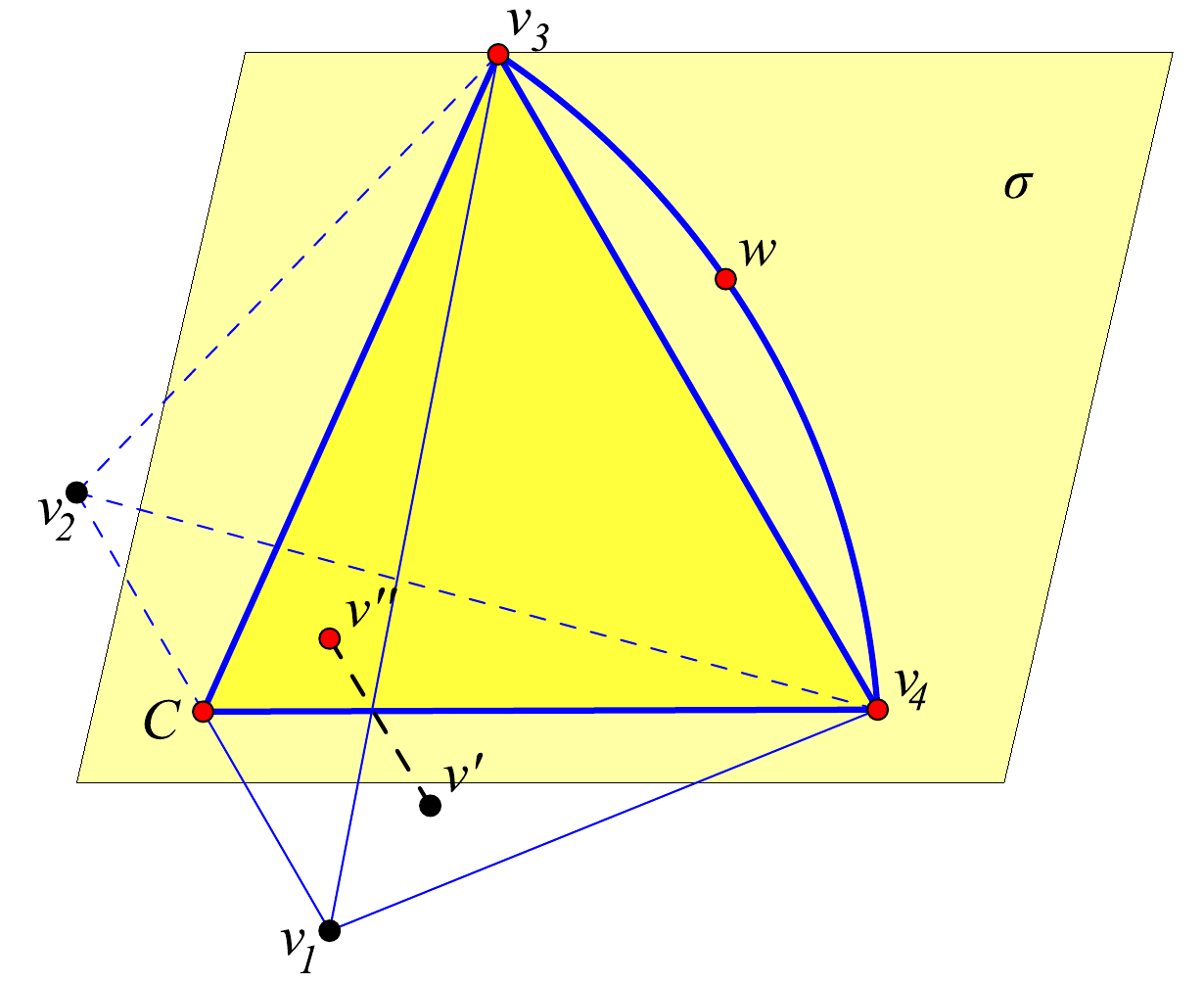}  \end{center}\begin{center}  Figure 6. $v'$ projects inside the convex hull of $C,v_3,v_4$. \end{center}

We first find an open spherical region $\Omega$ that contains $w$ and the sphere $U$ of minimal dimension that contains $\Omega$. We denote its center by $C$. Then we project $T$ on the flat  $\sigma$ of minimum dimension that contains $U$. Assume that $\sigma$ is an affine hull of the points $O, v_{k+1},\ldots, v_d$ (see Lemma \ref{lem02}). Then $C = \nu(v_1+\ldots, +v_k)/k$ with $\nu =1$ in the Euclidean case. The projection is fairly simply arranged (see Fig. 6, in which we show $T$, the face $\Omega$ of $\Delta_{\pi}$ that contains $w$ and the plane $\sigma$; $v_1,v_2$ are projected into $C$, while $v_3,v_4\in\sigma$; all points lying in $\sigma$ are marked by red). Any of the points $v_1,\ldots, v_k$ project into $C$. This is true due to the fact that for any $i=1,\ldots, k$ we have $S_i\cap\sigma = U.$  To the contrary, all the points $v_{k+1},\ldots, v_d$ project to themselves, since they lie in $\sigma.$ It is clear that the projection $v''$ of $v'$ falls into the projection of $T$.

In $\sigma$ consider a hyperplane $\omega$ which passes through $C$ and such that the vertices $w, v''$ and the whole projection of $T$ lie in the halfspace $\omega^+$ (recall that $w\in \Delta_{\pi}$). Any hyperplane in $\sigma$ that touches the set $T\cap\sigma$ in a single point $C$ would fit, and we obviously have such hyperplanes.

It is possible to apply Lemma \ref{lemred} and find a point $w''$ on an open spherical region of a smaller dimension such that $\rho(v'',w'')\ge \rho(v'',w).$ Moreover, if $v'\notin \partial T,$ then $v''$ cannot coincide with $C$, which means that the inequality is strict. By Pythagoras' theorem we get $\rho(v',w'')\ge \rho(v',w)$ (and a strict inequality in the case when $v'\notin \partial T$). Reducing the dimension of the spherical region which contains the current image of $w$ step by step, we eventually arrive at a vertex of $T$, which concludes the proof of both Lemma \ref{lemrelo} and \ref{lemimp}.
\end{proof}
\section{Acknowledgement} We are grateful to Filip Mori\'c and J\'anos Pach for several interesting and stimulating discussions on the problem, Eyal Ackerman for pointing out the paper \cite{MaeS}, Ido Shahaf for his valuable remarks on the presentation of the paper and for providing us with Figure 4, to Arseniy Akopyan and Roman Karasev for valuable discussions and for bringing Kirszbraun's theorem to our attention, and anonymous referees for numerous valuable comments that helped to significantly improve the presentation of the paper.

\end{document}